\documentclass[letterpaper, journal,onecolumn,draftclsnofoot]{IEEEtran}
\pdfoutput=1
\usepackage{ifpdf}
\IEEEoverridecommandlockouts
\usepackage{amsfonts}       % blackboard math symbols
\usepackage{bm}
\usepackage{amsmath,amssymb}
\usepackage{xcolor}
\usepackage{color}
\usepackage{graphicx}
\usepackage{caption}

\usepackage{enumitem}
\usepackage{amsthm}
\usepackage{float}
\usepackage{subcaption}
\newtheorem{theorem}{Theorem}
\newtheorem{proposition}{Proposition}

\newtheorem{definition}{Definition}
\newtheorem{lemma}{Lemma}

\newtheorem{assumption}{Assumption}

\graphicspath{ {./images/} }

\newcommand{\ymax}{y_{\max}}

\title{\LARGE \bf Multi-Agent Reinforcement Learning in Cournot Games}

\author{
	\IEEEauthorblockN{Yuanyuan Shi, Baosen Zhang}\\
	\IEEEauthorblockA{Department of Electrical and Computer Engineering\\University of Washington, Seattle, WA, USA
		\\\{yyshi, zhangbao\}@uw.edu}
	\thanks{This work is partially supported by the National Science Foundation awards CNS-1931718 and ECCS-1807142.}
}

\begin{document}
\maketitle

\begin{abstract}
In this work, we study the interaction of strategic agents in continuous action Cournot games with limited information feedback. Cournot game is the essential market model for many socio-economic systems where agents learn and compete without the full knowledge of the system or each other. We consider the dynamics of the policy gradient algorithm, which is a widely adopted continuous control reinforcement learning algorithm, in concave Cournot games. We prove the convergence of policy gradient dynamics to the Nash equilibrium when the price function is linear or the number of agents is two. This is the first result (to the best of our knowledge) on the convergence property of learning algorithms with continuous action spaces that do not fall in the no-regret class.
\end{abstract}

\section{Introduction}
\label{sec:intro}
Reinforcement Learning (RL) has yielded impressive results in various
sequential decision-making problems in recent years.
These successes include playing games with super-human performance~\cite{mnih2015human, brown2019superhuman}, solving complex robotic tasks~\cite{lillicrap2016continuous,schulman2015trust, mankowitz2020robust} and autonomous driving~\cite{shalev2016safe,sallab2017deep}. Some of these applications focus on a single agent, but many applications of interest consider groups of agents (or players). In the latter case, agents operate in a common environment, each of them interacting with the environment and other agents. This multi-agent setting contains a rich set of models and has received significant attention in the past several years (see~\cite{zhang2019multi} and references within).

In this paper, we study the dynamics of learning agents, where each agent aims to optimize its long-term expected return by repeatedly participating in the game. This question has been mostly studied at two extremes, where the agents are either fully cooperative~\cite{oroojlooyjadid2019review} or they are fully competitive (i.e. zero-sum games)~\cite{littman1994markov,schafer2019competitive}. Instead of these extremes, we focus on the case of \emph{general-sum games}, where the agents are self-interested but no adversarially so. General-sum games have been widely used to model the interactions and competition in cyber-physical systems because of the wide range of individual goals and possible relationships between agents~\cite{crandall2005learning,dixit2015games}. Each agent in the games is self-interested, and reward may be conflicting with others, but often not in a zero-sum manner. Compared to the two extreme cases, there have been relatively few convergence results for general-sum games, partly because of the technical challenge caused by heterogeneous goals and limited information.

We focus on a specific class of general-sum games where the agents undergo \emph{Cournot competitions}~\cite{cournot1838recherches}. Cournot game has been used to model the energy systems~\cite{kirschen2018fundamentals}, transportation networks~\cite{bimpikis2019cournot} and healthcare systems~\cite{chletsos2019hospitals}. It is also one of the most prevalent models of firm competition in economics. In the Cournot game model, firms control their production level, which influences the market price. For example, some electricity markets can be thought of as a Cournot game, where the production level is the amount of power produced by the generators, and the price is decided by the total generation bid and the demand. Each generator’s payoff is then calculated as the market price multiplying its share of the supply, subtracting its production cost~\cite{kirschen2018fundamentals}.\footnote{In this a first-order approximation of the locational marginal pricing used by markets in the United States.} In this example the generators do not cooperate with each other, but the total profit is also not zero.

When learning is not needed, there is a wealth of results for the Cournot competition. For example, when each agent has full information about the game, including the price function and the cost function of all other agents, there are many works characterizing the properties of the Nash equilibrium of the game~\cite{shapiro1989theories,kannan2012distributed,daughety2005cournot}.
However, when learning is involved and agents do not have full information, the properties of the game are not well understood. This is even the case in the simplest setting, where the agents only receive the price from the system as the feedback but do not know the price function form nor the actions of other agents.

To answer what happens when agents learn, we must model how they learn - or more precisely, what type of learning algorithms is used. A key technical challenge is that when learning is used, the Cournot game becomes stochastic.
Currently, most works focus on no-regret algorithms~\cite{nadav2010no,mertikopoulos2019learning,shi2019learning} because they only require a minimal set of assumptions on the game. In addition, the no-regret definition could be directly translated to the coarse correlated equilibrium condition~\cite{roughgarden2010algorithmic} for a wide range of algorithms (e.g., multiplicative-weight~\cite{arora2012multiplicative}, online mirror descent~\cite{hazan2016introduction}, Follow-the-Regularized-Leader~\cite{kalai2005efficient}). However, while the theoretical properties of no-regret algorithms are attractive, they also limit the applicability of these algorithms. In practice, systems and agents are often \emph{not adversarial} to each other, and the competition is often designed to have specific structures. In many games, it is more natural for players to use myopic policies such as reinforcement learning algorithms that directly aim for profit maximization~\cite{chen2018optimal}. In addition, the notation of coarse correlated equilibrium can be quite weak, and sharper results are often desired.

These algorithms can lead to much better performances than no-regret algorithms, but proving their convergence has proven to be challenging~\cite{zhang2019multi} since the coupling between the (continuous) actions of the players must be carefully analyzed. Attempts have been made to discretize the space(e.g., Q-learning~\cite{leslie2005individual}) then studying the resulting discrete game, but the dimensionally quickly grows and important features (e.g., convexity) are hard to retain~\cite{rodrigues2009dynamic, arslan2016decentralized}.

In this work, we directly work with the continuous action and state space by considering agents use policy gradient learning algorithms. In particular, we assume the class of policies where the actions are parameterized by the mean of distributions (e.g., Gaussian policies). The \emph{major contribution} of this work is in the following: we prove that when the price function is linear or when there are two agents, there is a unique Nash equilibrium (NE) in the \emph{stochastic} Cournot game, and the policy gradient converges exponentially quickly to the NE. This is the first result (to the best of our knowledge) on the convergence property of algorithms with continuous action spaces that do not fall in the no-regret class.

The rest of the paper is organized as follows. Section \ref{sec:model} covers the background and prior works on standard Cournot games and policy gradient algorithm. Section \ref{sec:theorem} describes the stochastic Cournot game and sketches the convergence proof for policy gradient agents in Cournot games. Section~\ref{sec:proof} provides detailed proof of the convergence result. We provide several case studies in Section~\ref{sec:results} that demonstrate the convergence behavior of agents in various settings. Finally, Section \ref{sec:conslusion} concludes the paper and outlines directions for future work.

\section{Problem Setup and Preliminaries}
\label{sec:model}
\subsection{Cournot Game}
\begin{definition}[Cournot Game]
	Consider $N$ players produce homogeneous products in a limited market, where the action space of player $i$ is its production level $x_i \geq 0$. The utility function of player $i$ is denoted as $\pi_i(\bm{x}) = p(\sum_{j=1}^{N} x_j)x_i-C_i(x_i)$, where $p$ is the market price (inverse demand) function that maps the total production quantity to a price in $\mathbb{R}$ and $C_i(\cdot)$ is the cost function of player $i$.
\end{definition}

The goal of each player $i$ in the Cournot game is to choose the best production quantity $x_i$ such that maximizes his utility $\pi_i$. An important concept in game theory is the \emph{Nash equilibrium}, at which state no player can increase their payoffs by unilaterally changing their strategies.
A Nash equilibrium of the Cournot game defined by $(\pi_1, ..., \pi_N)$ is a vector $\bm{x}^{*} \geq 0$ such that for all $i$:
\begin{equation}
\pi_i(x_i^{*}, \bm{{x}_{-i}^{*}}) \geq \pi_i(\tilde{x}_i, \bm{{x}_{-i}^{*}}), \text{\ \ for all\ \ } \tilde{x}_i,
\end{equation}
where $\bm{x}_{-i}$ denotes the actions of all players except $i$.
In this paper, we restrict our attention to Cournot games satisfying the following assumptions:
\begin{assumption}
	We assume the price function and cost functions:
	\begin{enumerate}
		\item[(A1)] The price function $p$ is concave, strictly decreasing and twice differentiable on $[0,\ymax]$, where $\ymax$ is the first point where $p$ becomes $0$. For $y>\ymax$, $p(y)=0$. In addition, $p(0)>0$. 
		\item[(A2)] The cost function $C_i(x_i)$ is convex, strictly increasing, twice differentiable and $p(0)>C_i'(0)$, for all $i$.
	\end{enumerate}
\end{assumption}
These assumptions are standard in the literature  (e.g., see~\cite{johari2005efficiency} and references within). The assumption $p(0)>C_i'(0)$ is to avoid the triviality of a player never participating in the game.  The following proposition shows that Cournot game satisfying the above assumptions has an unique Nash equilibrium.
\begin{proposition}\label{prop:unique_ne}
	A Cournot game satisfying (A1) and (A2) has exactly one Nash equilibrium.
\end{proposition}

Proof of Proposition~\ref{prop:unique_ne} refers to Theorem 1 in~\cite{szidarovszky1977new}.

\subsection{Policy-based Reinforcement Learning}
\label{section:pg}
In this work, we adopt a \emph{policy-based} methods of how agents would learn and act. For each agent, we assume that it has (possibly noisy) information of the system states at time $t$, which we denote by $\bf{s}_t$. This agent maintains a policy $\pi_{\theta}(\cdot|\bf{s}_t)$, which is a probability distribution on the action it would take, conditioned on the agent’s information $\bf{s}_{t}$. At each time step, after the agent picks actions $\bm{a}_t \sim \pi_{\theta}(\cdot|\bf{s}_t)$, the system releases reward.  Subsequently, players update their policy parameters along the gradient direction of their long-term expected reward. Such learning procedure is called \emph{policy gradient} method~\cite{sutton2000policy} in the literature. As a key permise for the idea, the policy long-term reward is,
\begin{equation}\label{eq:pg_pay}
J(\theta) = E_{\tau \sim p_{\theta}(\tau)} [\sum_{t} r(\mathbf{s}_t, \mathbf{a}_t)]
\end{equation}
and the gradient is given by,
\begin{equation}\label{eq:pg_gradient}
\nabla_{\theta} J(\theta) = E_{\tau \sim p_{\theta}(\tau)} [(\sum_{t=1}^{T} \nabla_{\theta} \log \pi_{\theta}(\bm{a}_t|\bm{s}_t))  (\sum_{t=1}^{T} r(\bm{s}_t, \bm{a}_t))]\,,
\end{equation}
where $\tau$ and $J(\theta)$ are trajectories and the expected trajectory return under policy $\pi_\theta$, respectively, and $\nabla_{\theta} \log \pi_{\theta}(\bm{a}_t|\bm{s}_t))$ is the score function of the policy. Various of policy gradient methods have been proposed by estimating the gradient~\eqref{eq:pg_gradient} in different ways, including REINFORCE~\cite{sutton2000policy}, natural policy gradient~\cite{kakade2002natural} and actor-crtic algorithms~\cite{mnih2016asynchronous}.

% \newpage
\section{Stochastic Cournot Game}
\label{sec:theorem}
We discuss the main convergence results in this section. As briefly mentioned, we consider policy-based models of how agents choose and evolve their actions. In particular, as the agents only get the reward as feedback and nothing else, the dynamics in Section~\ref{section:pg} reduces to the \emph{stateless} version. This is consistent with the practice of many social-economic systems (e.g., energy markets~\cite{quinn2009privacy}), where providing full feedback is either impractical or explicitly disallowed due to privacy and market power concerns.

In particular, we consider a policy that is parameterized by the mean of a distribution. This model includes many popular algorithms, for example, the ubiquitous Gaussian policies and their extensions~\cite{chou2017improving}.
Let $\theta_i$ denote the \emph{mean} of player $i$'s action, and $X_i$ to be a zero-mean random variable. For convenience, we assume it is continuous and has a bounded density function denoted by $f_i(X_i)$.
We say $X_i$ is unimodal at mean if $f_i$ has a global maximum at the mean and no other isolated local maxima \footnote{For example, Gaussian and uniform distributions are unimodal under this definition.}.

At each time step, player $i$ choose the action to play as $a_i \sim \pi_{\theta_i}(\cdot)= \theta_i+X_i$. Note that in most Cournot games, the action is interpreted as quantity, that cannot be negative.
Therefore, player $i$ has to play by drawing a quantify from the rectified distribution $(\theta_i+X_i)^+$, where $a^+=\max(a,0)$.
Under the Cournot game setup, the expected profit in Eq.~\eqref{eq:pg_pay} can be written out as the follows,
\begin{equation}\label{eqn:pg_cournot_pay}
\begin{split}
J_i(\theta_i;\bm{\theta}_{-i})=E_{\bm{X}}\left[p\left(\sum_{j=1}^N(\theta_j+X_j)^+ \right) (\theta_i+X_i)^+\right.\\
\left.- C_i((\theta_i+X_i)^+)\right].
\end{split}
\end{equation}
and the gradient value in Eq.~\eqref{eq:pg_gradient} equals,
\begin{small}
	\begin{align}\label{eqn:pg_cournot_gradient}
	\nabla_{\theta_i} J_i &= E\left[ 1(\theta_i+X_i \geq 0) \left\{p'\left(\sum_{j=1}^N(\theta_j+X_j)^+ \right) (\theta_i+X_i)\right.\right. \nonumber\\
	&\left.\left.+p\left(\sum_{j=1}^N(\theta_j+X_j)^+ \right)  -C_i'(\theta_i+X_i) \right\}\right],
	\end{align}
\end{small}
where $1(\cdot)$ is the indicator function. 

We call the game associated with these $J_i$'s the \emph{stochastic Cournot game}, where player $i$ chooses $\theta_i$, observe the profit $J_i$, and update $\theta_i$ according to the payoff gradient. The Nash equilibrium of the stochastic Cournot game is defined as, $(\theta_1^{*}, ..., \theta_N^{*})$ such that $\nabla_{\theta_i^{*}} J_i = 0, \forall i$. 
The form of \eqref{eqn:pg_cournot_gradient} has made analyzing the system dynamics difficult compared to standard Cournot games. Firstly, because the actions are rectified, the profit of player $i$ not long just depends on the sum of the other players (as in a Cournot game), but it actually depends on each of the other players' parameters. This rules out many elegant and simple results on the existence and uniqueness of Nash equilibria~\cite{szidarovszky1982contributions,novshek1985existence,amir1996cournot,ewerhart2014cournot}. Secondly, although the realization fo the actions are nonnegative, it is not obvious that $\theta_i$'s need to be nonnegative, or even bounded. The main result in the paper is to overcome these challenges and show that under some assumptions, the game is well-behaved and policy gradient updates converge exponentially quickly to the Nash equilibrium in Stochastic Cournot games.
\begin{theorem}\label{theorem: pg_conv}
	Consider a stochastic Cournot game satisfying the assumptions (A1) and (A2). Suppose each player's policy is parameterized as the mean $\theta_i$ and a zero-mean random variable $X_i$ that is unimodal at the mean with infinite support, and suppose that all players follow policy gradient in \eqref{eqn:pg_cournot_gradient} to update their mean. Then the policies converge to the Nash equilibrium exponentially quickly for all initializations either of the following condition holds:
	\begin{enumerate}
		\item The price function is linear.
		\item The number of players equals two.
	\end{enumerate}
	% has a unique Nash equilibrium. Suppose each player's policy is parameterized as the mean $\theta_i$ and a zero-mean random variable $X_i$  (unimodal at the mean with infinite support). If all players follow the policy gradient algorithm for the mean update, the expected sequence of play converges to the Nash equilibrium exponentially quickly for all initializations under two conditions: 1) price function is linear; or 2) the number of players equals two.
\end{theorem}

The condition of the theorem includes Gaussian policies, which is a natural choice for continuous action spaces~\cite{sutton2000policy,silver2014deterministic}, and such a form also includes popular neural network policies~\cite{levine2013guided} where the mean can be parameterized via a neural network.  We also do not restrict the players to be symmetric, and each of the players would adopt different variances or even have completely different classes of distributions. The infinite support requirement of the distribution is a technicality and can be weakened, although it would make the proofs much more cumbersome.

The proof of Theorem~\ref{theorem: pg_conv} proceeds in three lemmas. We defer the full proofs of these lemmas to Section~\ref{sec:proof} and sketch the steps in the proof here. The first step in the proof is to show that we can restrict the actions of the players to a compact region using the following lemma:
\begin{lemma}\label{lem:bound}
	Under the assumptions of Theorem~\ref{theorem: pg_conv}, $\theta_i$ can be restricted to $[\underline{\theta}_i,\ymax]$, where $\underline{\theta}_i$ is a constant.
\end{lemma}

This lemma essentially confines the choices of the players to a compact interval, which sets up the rest of the proof. As a reminder, $\ymax$ is the point where the price function becomes $0$. The proof of this lemma is based on showing that player $i$'s profit will be suboptimal if it chooses an $\theta_i$ outside of the interval, regardless of other players' choices.

Interestingly, to show the parameters of the policy gradients converges to the Nash equilibrium of the stochastic Cournot game for the two cases stated in Theorem~\ref{theorem: pg_conv}, we need two different proof techniques. Therefore, we separate them into two lemmas as stated below.
\begin{lemma}\label{lem:NE_conv1}
	Under the assumptions of (A1)-(A2) and suppose the market price is linear, the policy gradient updates converge to the unique Nash equilibrium exponentially fast under all initial conditions.
\end{lemma}
\begin{lemma}\label{lem:NE_conv2}
	Under the assumptions of (A1)-(A2) and suppose there are only two players, the policy gradient updates converge to the unique Nash equilibrium exponentially fast under all initial conditions.
\end{lemma}

The proof of Lemma~\ref{lem:NE_conv1} leverages Rosen's  conditions in ~\cite{rosen1965existence}. A sufficient condition for the convergence of gradient-based algorithms in concave N-player games is that the game Hessian is \emph{negative definite}.
% Let $\bm{G}$ be the Hessian of the payoff functions, that is, $G_{ij}=\frac{\partial^2 J_i}{\partial \theta_i \partial \theta_j }$.
Therefore, we prove Lemma~\ref{lem:NE_conv1}  in two steps. First is to show that the stochastic Cournot games with assumptions of (A1)-(A2) are concave N-player game, and then show the game Hessian is negative definite under linear price functions.
However, once the price function is not linear, we cannot directly use Rosen's conditions for the convergence proof, even under the two-player case. The proof of Lemma~\ref{lem:NE_conv2} is based on a dynamical system interpretation. We proved that under the two-player general price function setup, the game Hessian is strictly diagonally dominant with all eigenvalues in the left-half plane, thus the Nash equilibrium is an exponentially stable fixed point.

We close this section with two remarks. Firstly, our proof provides the sufficient conditions for the convergence of policy gradient in Cournot games, that is either the price function is linear, or the player number is no more than two for general price function. However, these may not be necessary. We provide a three-player example with quadratic price function in Section~\ref{sec:investigative}, where we also observe convergence behavior. Secondly, it should be noted that in practice, some players may decide to not follow the policy gradient updates and use other learning algorithms (or act in adversarial manners). We provide some empirical evaluations of the system robustness in Section~\ref{sec:investigative}, by assuming a small portion of players is acting randomly. Both directions, 1) generalizing the convergence proof to a broader class of games and 2) dynamics under heterogeneous/adversarial learning agents are important as future works.

\section{Proof of Theorem~\ref{theorem: pg_conv}}
\label{sec:proof}
In this section, we prove the three major lemmas stated in the previous section.
\subsection{Proof of Lemma~\ref{lem:bound}}
Without loss of generality, we can consider player 1. Fix the other player's choices of $\theta$'s. Define the random variable $Y=\sum_{i=2}^N (\theta_i+X_i)^+$. We first prove that it is never beneficial for player 1 to set $\theta_1$ to a value larger than $\ymax$. Consider the derivative of $J_1$ with respect to $\theta_1$
\begin{align}
	g_1&= \frac{\partial }{\partial \theta_1} E\left[p\left((\theta_1+X_1)^+ + Y\right) (\theta_1+X_1)^+-C_1\left((\theta_1+X_1)^+\right)\right] \nonumber\\
	= & E\left[ 1(\theta_1+X_1 \geq 0) \left\{p'\left(\theta_1+X_1 + Y\right) (\theta_1+X_1)\right.\right. \nonumber\\
	&\left.\left.+p\left((\theta_1+X_1)^+ + Y\right) -C_1'(\theta_1+X_1) \right\}\right], \label{eqn:gi}
\end{align}

where $1(\cdot)$ is the indicator function. We want to show that if $\theta_1\geq \ymax$, the derivative is negative. The last term $-E\left[1(\theta_1+X_1 \geq 0)C_i'(\theta_1+X_1)\right]$ is negative because $C_i$ is strictly increasing. Now consider the first two terms, and let $f_Y$ be the density of $Y$,
\begin{align}
& E\left[ 1(\theta_1+X_1 \geq 0) \left\{p'\left(\theta_1+X_1 + Y\right) (\theta_1+X_1) +p\left(\theta_1+X_1 + Y\right) \right\}\right] \nonumber \\
&=  \int_0^\infty \int_0^\infty 1(\theta_1+x_1 \geq 0) \left\{p'\left(\theta_1+x_1 + y\right) (\theta_1+x_1)+p\left(\theta_1+x_1 + y\right) \right\}f_1(x) f_Y(y) dx dy \nonumber \\
&\stackrel{(a)}{=} \int_0^{\ymax} \int_{-\theta_1}^{\ymax-y-\theta_1} \left[p'\left(\theta_1+x_1 + y\right) (\theta_1+x_1)+p\left(\theta_1+x_1 + y\right)\right] f_1(x) f_Y(y) dx dy  \nonumber \\
&\stackrel{(b)}{=} \int_0^{\ymax} \int_0^{\ymax-y} \left[p'\left(x_1' + y\right) (x_1')+p\left(x_1' + y\right)\right] \cdot f_1(x'-\theta_1) f_Y(y) dx dy, \label{eqn:int_1}
\end{align}
where $(a)$ follows from assumption (A1) and $(b)$ from a change of variable from $x_1$ to $x_1'=\theta_1+x_1$. Next we show that for any given $y$, $\int_0^{\ymax-y} p'\left(x_1' + y\right) (x_1')+p\left(x_1' + y\right) dx=0$. Using the integration by parts on the first term, denoting $\bar{y} =\ymax-y$, we have
\begin{align*}
& \int_0^{\ymax-y} p'\left(x_1' + y\right) (x_1')+p\left(x_1' + y\right) dx \\
= & p(x_1'+y)x_1' \vert_{x_1'=0}^{x_1'=\bar{y}}- \int_0^{\bar{y}} p\left(x_1' + y\right) + \int_0^{\bar{y}} p\left(x_1' + y\right) \\
= & 0.
\end{align*}
By assumption (A1), $p'\left(x_1' + y\right) (x_1')+p\left(x_1' + y\right)$ is positive at $x_1'=0$. Therefore, it must undergo a sign change from positive to negative. However, by the unimodality assumption, $\theta_1 \geq \ymax$, $f_1 (x_1-\theta_1)$ is an increasing function on the interval $x_1 \in [0,\ymax-y]$. Therefore,  $\int_0^{\ymax-y} \left[p'\left(x_1' + y\right) (x_1')+p\left(x_1' + y\right)\right] f_1(x'-\theta_1) dx<0$ for all $y$ and this proves that player $1$ would never choose $\theta_1$ to be larger or equal to $\ymax$.

Now we show that there is a lower bound on $\theta_1$. The partial derivative of $g_1$ with respect to $\theta_j$ is
{\begin{small}
	\begin{align*}
	\frac{\partial g_1}{\partial \theta_j} &= E \left[ 1(\theta_1+X_1 \geq 0, \theta_j+X_j \geq 0) \cdot \left\{p''\left(\sum_{j=1}^N (\theta_j+X_j)^+\right)  \cdot (\theta_1+X_1)   +p'\left(\sum_{j=1}^N (\theta_j+X_j)^+\right) \right\}\right] < 0
	\end{align*}
\end{small}}
where the inequality follows from $p$ is strictly decreasing and concave. Similar calculations can be used to show that all cross partials are negative. Therefore, player $1$ should decrease $\theta_1$ as other players increase their parameters. From the first part of the proof, suppose all other players choose $\ymax$ as their play. Even at this choice, $g_1$ still becomes positive for negative enough $\theta_1$'s, therefore implying that the choice of $\theta_1$ is lower bounded by some real number $\underline{\theta}_1$.

\subsection{Proof of Lemma~\ref{lem:NE_conv1}}
Lemma~\ref{lem:bound} shows that the action space is convex and compact. Let $\bm{G}$ denote the Hessian of the game, so
\begin{equation*}
G_{ij}=\frac{\partial^2 J_i}{\partial \theta_i \partial \theta_j}= \frac{\partial g_i}{\partial \theta_j}.
\end{equation*}
Focusing on the diagonal terms, we have
\begin{small}
	\begin{align}
	G_{ii} &=\frac{\partial^2 J_i}{\partial \theta_i^2} = E \left[ 1(\theta_i+X_i \geq 0) \cdot \left\{p''\left(\sum_{j=1}^N (\theta_j+X_j)^+\right) \right.\right.\nonumber\\
	& \left.\left. \cdot (\theta_i+X_i)  +2p'\left(\sum_{j=1}^N (\theta_j+X_j)^+\right) -C_i''(\theta_i+X_i) \right\}\right] < 0,
	\end{align}
\end{small}
which is negative by assumptions (A1) and (A2). A game is said to be a \emph{concave N-player game}~\cite{rosen1965existence} if $G_{ii}<0, \forall i$ and the action space is convex and compact. Therefore, it is a concave N-player game. The following proposition is given in~\cite{rosen1965existence} as a sufficient condition to show when gradient-based algorithms converge to Nash equilibriums:
% In addition,~\cite{rosen1965existence} gives the sufficient convergence condition for gradient-based algorithms in concave N-player games. Let $\bm{G}$ be the game Hessian, that is, $G_{ij}=\frac{\partial^2 J_i}{\partial \theta_i \partial \theta_j }$, and we have
\begin{proposition}\label{prop: rosen_condition}
	Let $\bm{G}$ denote the Hessian of a concave N-player game. If $\bm{G}^T+\bm{G}$ is negative definite over the space of actions, there is a unique Nash equilibrium and the policy gradient dynamics approach it exponentially quickly for all initializations.
\end{proposition}

Using Proposition~\ref{prop: rosen_condition}, it suffices for us to show that the negative definiteness of $\bm{G}^T+\bm{G}$ under the stochastic Cournot game. The main challenge of proving the negative definiteness lies in the \emph{expectation} term, and we need to relate the properties of the Hessian of a function to its expectation.  The following proposition tackles the aforementioned challenge and relates the Hessian property to its expectation.
\begin{proposition}\label{prop:nd}
	Let $f:\mathbb{R}^N\rightarrow \mathbb{R}$ be a continuous function and suppose that the first and second order partial derivatives exist for all points except possibly for a set of measure $0$. Let $\bm{G}$ be the Hessian of $f$, whenever it exists. Now consider the function $\hat{f}:\mathbb{R}^N\rightarrow \mathbb{R}$, where $\hat{f}(y)=E_{\bm{X}}{f(y+\bm X)}$ and $\bm{X}$ is random vector in $\mathbb{R}^N$, with continuous and bounded density function and infinite support. Let $\hat{\bm{G}}$ be the Hessian of $\hat{f}$. Then: i) If $\bm{G}^T+\bm{G}$ is negative semidefinite at all points where $\bm{G}$ exists, then $\hat{\bm{G}}^T+\hat{\bm{G}}$ is negative semidefinite.
	ii) If $\bm{G}^T+\bm{G}$ is negative definite for a set of measure larger than 0, then $\hat{\bm{G}}^T+\hat{\bm{G}}$ is negative definite.
\end{proposition}
\begin{proof}
	The proof of this proposition is straightforward. By the assumption on the random vector $\bm X$, we can switch the order of differentiation and the expectation. In addition, the density being continuous allows us to ignore the points where $\bm{G}$ does not exist. Then
	\begin{align*}
	& \bm{v}^T (\bm{{G}}^T(\bm y)+\bm{{G}}(\bm y)) \bm {v} \\
	&= E_{\bm{X}} \left[\bm{v}^T (\bm{G}^T(\bm{y}+\bm X)+\bm H(\bm{y}+\bm X)) \bm {v}\right]  \leq 0,
	\end{align*}
	for any $\bm v$. Now suppose $\bm{G}^T(\bm{y}+\bm x)+\bm G(\bm{y}+\bm x))$ is negative definite for set of positive measure, then by the continuity of the density function, $ \bm{v}^T (\bm{{G}}^T(\bm y)+\bm{{G}}(\bm y)) \bm {v}<0$ for all nonzero $\bm v$ and $\hat{\bm{G}}^T+\hat{\bm{G}}$ is negative definite.
\end{proof}

Let $f_i(\bm{x})=p(\sum_i x_i^+) x_i^+$, which is continuous and twice differentiable except for a measure zero set on $\mathbb{R}^N$. Since $J_i=E[f_i(\bm{\theta}+\bm X)]$, we need to show $f=\begin{bmatrix} f_1 \dots f_N \end{bmatrix}$ satisfies the condition of Proposition \ref{prop:nd}. Given a vector $\bm x$, without loss of generality, assume that $x_1, \dots,x_k \geq 0$ and $x_{k+1},\dots,x_N<0$. The second order derivatives of $f$ are,
\begin{small}
	\begin{align*}
	\frac{\partial^2 f_i}{\partial x_i \partial x_j} = & 1\left(x_i \geq 0\right) \cdot
	\begin{cases} p''(\sum_l x_l^+) x_i + 2 p'(\sum_l x_l^+) - C_i''(x_i), i=j \\
	1(x_j \geq 0) \left(p''(\sum_l x_l^+) x_i + p'(\sum_l x_l^+)\right), i \neq j
	\end{cases}.
	\end{align*}
\end{small}

Because of the indicator on both $x_i \geq 0$ and $x_j \geq 0$, the Hessian is only nonzero for the \emph{upper left block}. In this block, we have $\forall i, j \leq k$,
\begin{equation*}
\frac{\partial^2 f_i}{\partial x_i \partial x_j} =
\begin{cases} p''(\sum_{l=1}^k x_l) x_i + 2 p'(\sum_{l=1}^k x_l)- C_i''(x_i), i=j\\
p''(\sum_{l=1}^k x_l) x_i +  p'(\sum_{l=1}^k x_l), i \neq j,
\end{cases}.
\end{equation*}

When the price function is linear, the second order derivative term vanishes, i.e. $p''(\sum_{l=1}^{k} x_l)  =0$. Therefore, we have,
\begin{equation*}
\frac{\partial^2 f_i}{\partial x_i \partial x_j} =
\begin{cases} 2 p'(\sum_{l=1}^k x_l)- C_i''(x_i) \mbox{ if } i=j, i, j \leq k \\
p'(\sum_{l=1}^k x_l) \mbox{ if } i \neq j, i, j \leq k
\end{cases}
\end{equation*}
Now we can write $\bm{G}$ as $\bm{G}_1+\bm{G}_2+\bm{G}_3$ where
\begin{small}
	\begin{align}\label{eq:g1}
	\bm{G}_1 = \left[
	\begin{array}{c|c}
	\begin{bmatrix}p'(\sum_{l=1}^k x_l) &  \cdots & p'(\sum_{l=1}^k x_l) \\
	\vdots & \ddots & \vdots \\
	p'(\sum_{l=1}^k x_l) &  \cdots & p'(\sum_{l=1}^k x_l)
	\end{bmatrix}& \begin{bmatrix} 0 & \cdots & 0 \\
	\vdots & \ddots & \vdots\\
	0 & \cdots & 0
	\end{bmatrix}\\
	\hline
	\begin{bmatrix} 0 & \cdots & 0 \\
	\vdots & \ddots & \vdots\\
	0 & \cdots & 0
	\end{bmatrix}& \begin{bmatrix} 0 & \cdots & 0 \\
	\vdots & \ddots & \vdots\\
	0 & \cdots & 0
	\end{bmatrix}
	\end{array}
	\right]
	\end{align}
\end{small}

Both $\bm{G}_2, \bm{G}_3$ are diagonal matrices. For $\bm{G}_2$, it has the $i$'th component being $p'(\sum_{l=1}^k x_l)$ for $i\leq k$ and $0$ for $i>k$. Since $X_i$ have infinite support, there exists cases where $k=N$ (all player sample non-negative actions) for a \emph{measure larger than 0 set}, in which $\bm{G}_2$ is \emph{negative definite}. For $\bm{G}_3$, it has the $i$'th component being $-C_i''(x_i) \leq 0$ for $i\leq k$ and $0$ for $i>k$, thus it is negative semi-definite. Therefore, it suffices for us to show the negative semi-definess of $\bm{G}_1$.

The eigenvalues of $\bm{G}_1$ in Eq.~\eqref{eq:g1} are the combination of eigenvalues of the upper left and lower-right matrices. Eigenvalues of the upper left matrix are $kp'(\sum_{l=1}^k x_l) <0$ and 0 ($k-1$ repeats), and eigenvalues of the lower right are all zeros. Thus $\bm{G}_1$ is negative semi-definite.

Therefore, there exists an measure nonzero set of actions, such that $\bm{G} = \bm{G}_1+\bm{G}_2+\bm{G}_3$ is negative definite. By Proposition~\ref{prop:nd}, we have the Hessian of $(J_1, ..., J_N)$, that is $\hat{\bm{G}}$ is negative definite.

\subsection{Proof of Lemma~\ref{lem:NE_conv2}}
Now, consider the two-player Cournot games with general price function $p(\cdot)$ under assumption (A1) and (A2).
There are four cases considering the positiveness of $x_1$ and $x_2$.
\begin{enumerate}[label=\alph*)]
	\item  $x_1, x_2 \geq  0$: $$\bm{G}_a = \begin{cases} p''(x_1 + x_2) x_i + 2p'(x_1 + x_2)-C_i''(x_i), i=j, \\
	p''(x_1 + x_2) x_i +  p'(x_1 + x_2), i \neq j,
	\end{cases}$$
	\item  $x_1 < 0, x_2 \geq  0$:
	$$\bm{G}_b = \begin{bmatrix}
	0 & 0\\
	0 & p''(x_1 + x_2) x_2 + 2p'(x_1 + x_2)
	\end{bmatrix}$$
	\item  $x_1 \geq 0, x_2 <  0$: $$\bm{G}_c = \begin{bmatrix}
	p''(x_1 + x_2) x_1 + 2p'(x_1 + x_2) & 0\\
	0 &  0
	\end{bmatrix}$$
	\item  $x_1 < 0, x_2 <  0$: $$\bm{G}_d = \begin{bmatrix}
	0 & 0\\
	0 & 0
	\end{bmatrix}$$
\end{enumerate}
The game Hessian matrix thus follows,
\begin{align}\label{eqn: hessian_twoplayer}
\hat{\bm{G}} &= E\left[1(x_1, x_2 \geq  0) \bm{G}_a + 1(x_1 < 0, x_2 \geq  0) \bm{G}_b \right. \nonumber\\
& \left.+ 1(x_1 \geq 0, x_2 <  0)\bm{G}_c + 1(x_1 < 0, x_2 <  0) \bm{G}_d \right]\,,
\end{align}

$\bm{G}_a$ is a strictly diagonally dominant matrix since the magnitude of the diagonal entry is strictly larger than the \emph{sum} of the magnitudes of all the other (non-diagonal) entries in each row, i.e., $|p''(x_1 + x_2) x_i + 2p'(x_1 + x_2)-C_i''(x_i)| > |p''(x_1 + x_2) x_i +  p'(x_1 + x_2)|, \forall i$. Given that $\bm{G}_b, \bm{G}_c, \bm{G}_d$ are all diagonally dominant matrices, $\hat{\bm{G}}$  in Eq.~\eqref{eqn: hessian_twoplayer} is a strictly diagonally dominant matrix .
Therefore, the eigenvalues of matrix $\hat{\bm{G}}$ are all in the left–half plane (i.e., the real parts of eigenvalues are negative) by the Gershgorin circle theorem~\cite{weisstein2003gershgorin}.
Proposition 4 in~\cite{ratliff2013characterization}  showed that when all eigenvalues of the Hessian are in the open left–half plane, then the Nash equilibrium is an exponentially stable fixed point of the dynamical system generated by the gradient descend algorithm.

\section{Numerical experiments}
\label{sec:results}
In this section, we exam the performance of policy gradient algorithms in various of Cournot games. We first verify the convergence behavior under linear price and two-player cases. Next, we provide investigative studies on the system behavior under multi-player and players with random actions scenarios is well as intuitions for the robustness behavior.
\subsection{Experiment Setup}
We perform all the experiments using the natural policy gradient algorithm~\cite{kakade2002natural} with a Gaussian policy.  Following the derivations in Section~\ref{section:pg}, the gradient with respect to the policy parameter $\theta_i$ follows,
\begin{align}
\nabla_{\theta_i} J_i(\theta_i) & =E_x[\pi_i(x_i, x_{-i}) \nabla_{\theta_i} \log f_{\theta_i} (x_i)] \nonumber\\
& = \frac{1}{N} \sum_{i=1}^{N} \hat{\pi}_i \nabla_{\theta_i} \log f_{\theta_i} (x_i)\,,
\end{align}
where $\pi_i(x_i, x_{-i})$ is the payoff function of player $i$ and $\hat{\pi}_i$ is the observed payoff.
In the above formula, $f_{\theta_i}(x_i) = \theta_i + X_i$ is the decision making policy for player $i$ and $X_i \sim N(0, \sigma_i)$. For the action, we have $x_i = (\mu_i + X_i)^{+}$, where the action is truncated to be non-negative. The update rules for $\mu_i$ follows the natural policy gradient in ~\cite{kakade2002natural}.
%\begin{align}
%\mu_{i, t+1} &= \mu_{i, t} + \alpha_t \hat{\pi}_{i, t}(x_{i, t} - \mu_{i, t})\,,\\
%\sigma_{i, t+1} &=\sigma_{i, t} + \beta_t \hat{\pi}_{i, t} \frac{(x_{i, t} - \mu_{i, t})^2 - \sigma_{i, t}^2}{\sigma_{i, t}}.
%\end{align}
We choose the standard deviation for each player the same as $\sigma = 0.05$. All experiments are run using a 2.2 GHz Intel Core i7 Macbook Pro with 16 GB memory.

\subsection{Cournot Game Examples}
In this section, we verify the convergence behavior of the proposed algorithm in four example Cournot games, with different price and individual cost settings. $\textbf{G1:}$ three-player with linear price function $p(\bm{x}) = 1-(x_1+x_2+x_3)$ and no individual cost $C_i(x_i)=0, \forall i$. The Nash equilibrium is $x_1^{*} = x_2^{*} = x_3^{*} = \frac{1}{4}$. $\textbf{G2:}$ three-player with linear price function $p(\bm{x}) = 1-(x_1+x_2+x_3)$ and differnet individual cost $C_i(x_i)= 0.1  \cdot i \cdot x_i$ for player $i$. The Nash equilibrium is $x_1^{*} = 0.3, x_2^{*} = 0.2, x_3^{*} = 0.1$. $\textbf{G3:}$ two-player quadratic price function $p(\bm{x}) = 1-(x_1+x_2)^2$ without cost. The Nash equilibrium is $x_1^{*} = x_2^{*} = \sqrt{1/8} \approx 0.3536$. $\textbf{G4:}$ two-player cubic price function $p(\bm{x}) = 1-\frac{1}{2}(x_1+x_2)^3$ without cost. The Nash equilibrium is $x_1^{*} = x_2^{*} = \sqrt[3]{1/20} \approx 0.3684$.
In all of the games, each player simultaneously picks a production level. The price is determined by the sum of productions and broadcasted back to all players. This game is repeated multiple times with all players use policy gradient to learn and act. The dynamics of the policy parameter (i.e., the mean) are plotted in Figure~\ref{fig:convergence}. In all simulated games with different initializations and settings, the policy parameters converge to the Nash equilibrium, which verifies the theoretical results in Section~\ref{sec:theorem}.
\begin{figure}
	\centering
	\begin{subfigure}[b]{0.42\textwidth}
		\centering
		\includegraphics[width=\textwidth]{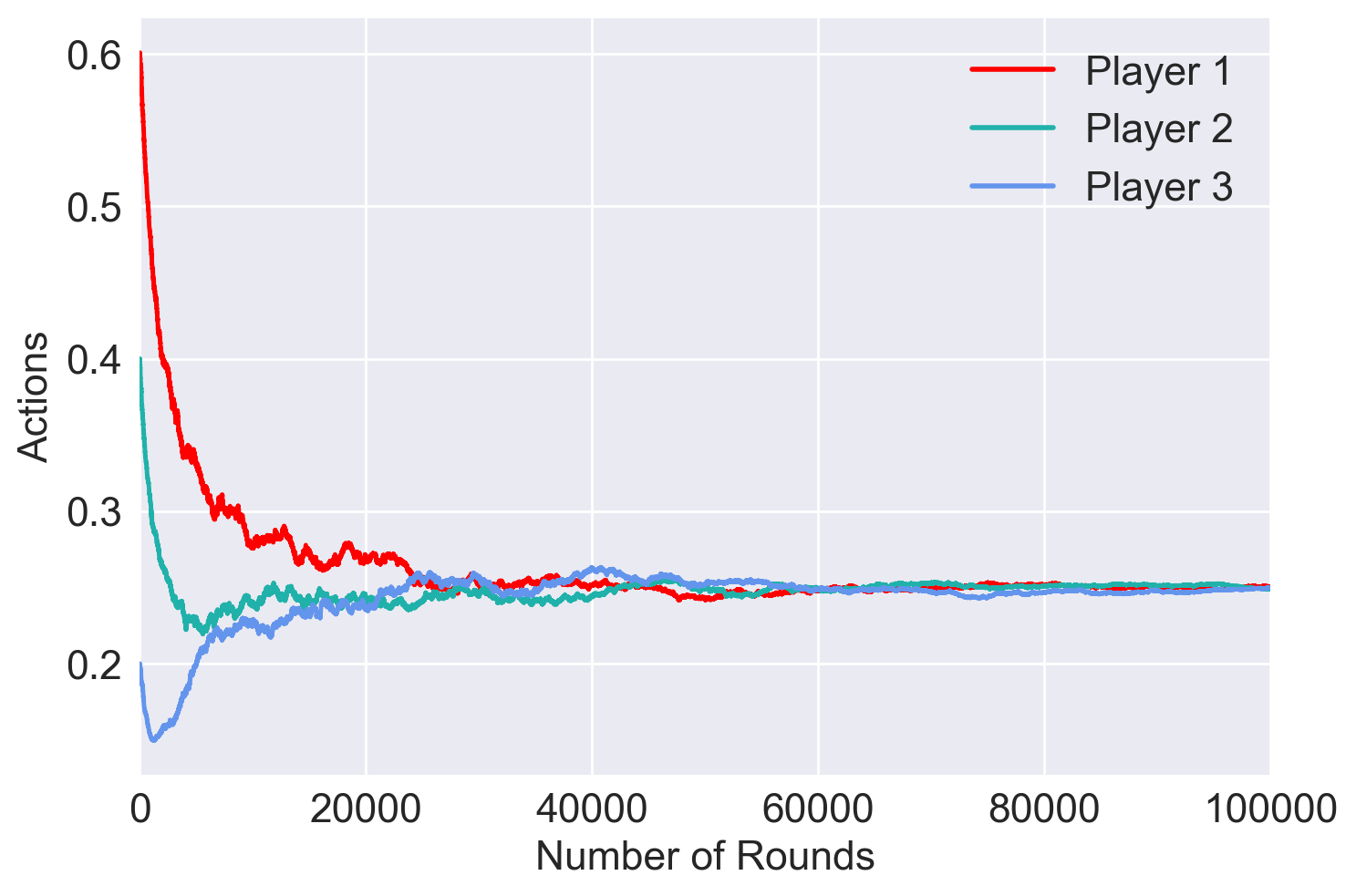}
		\caption{G1}
		\label{fig:g1}
	\end{subfigure}
	\begin{subfigure}[b]{0.42\textwidth}
		\centering
		\includegraphics[width=\textwidth]{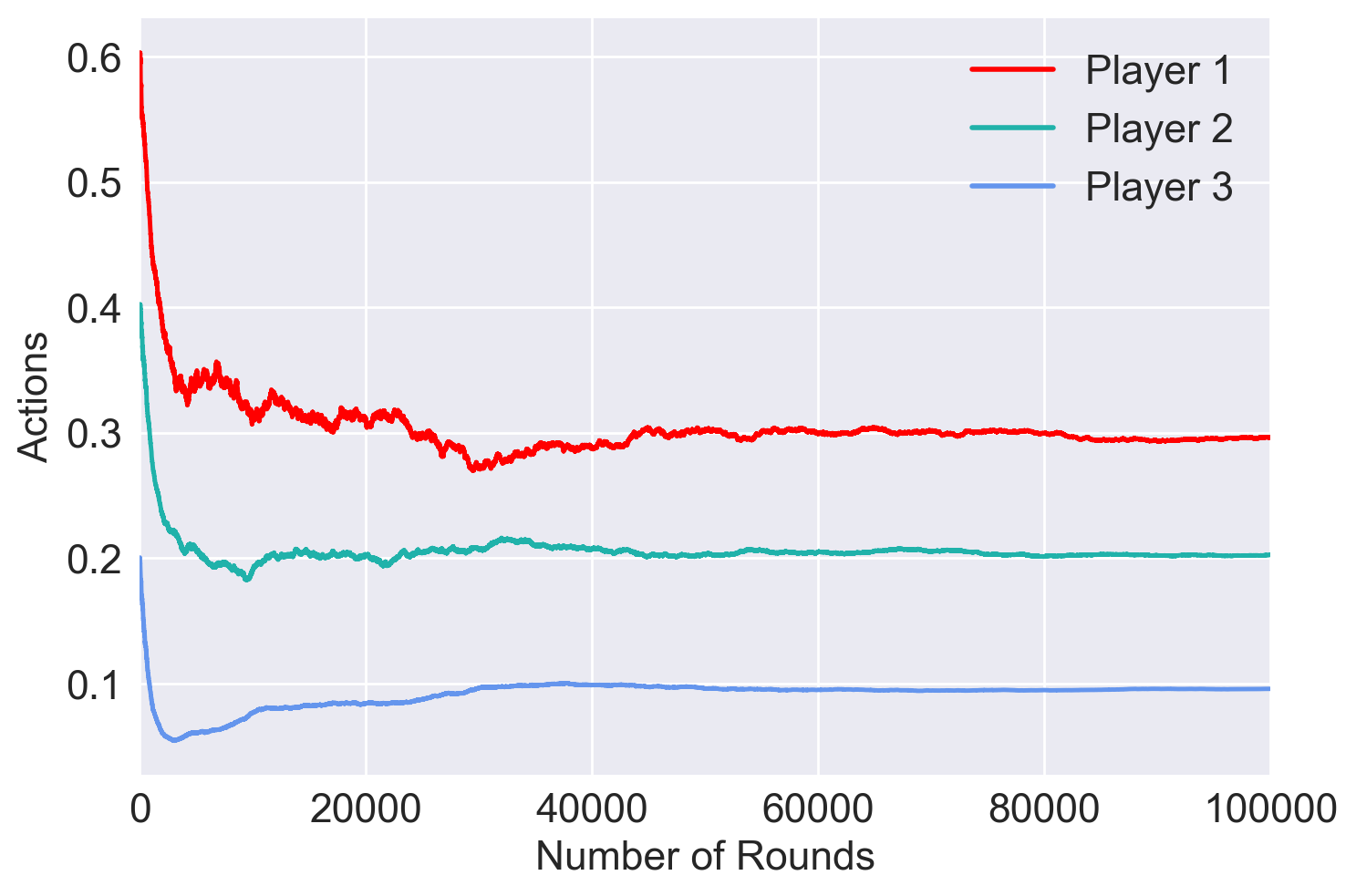}
		\caption{G2}
		\label{fig:g2}
	\end{subfigure}
	\hfill
	\begin{subfigure}[b]{0.42\textwidth}
		\centering
		\includegraphics[width=\textwidth]{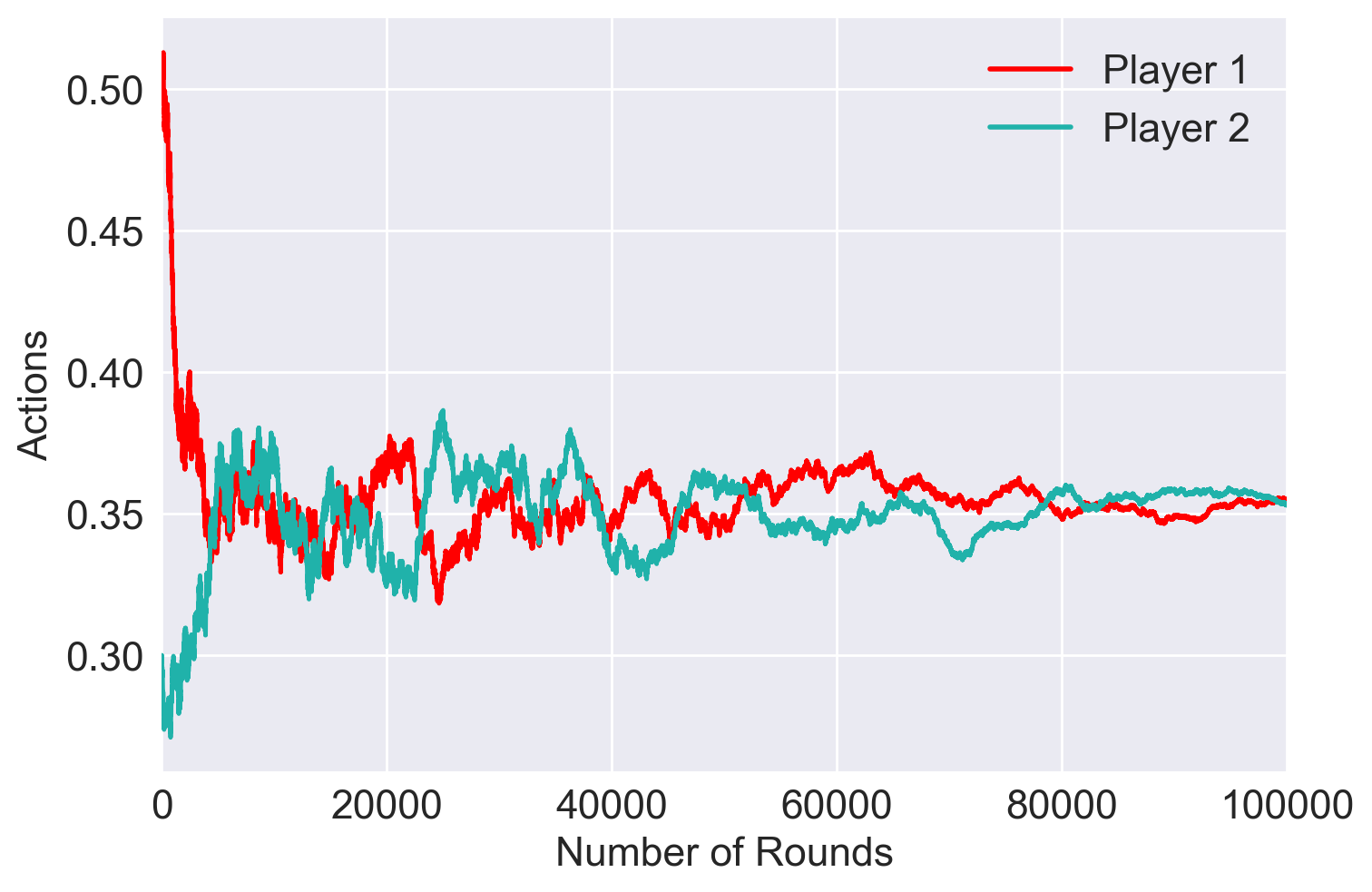}
		\caption{G3}
		\label{fig:g3}
	\end{subfigure}
	\begin{subfigure}[b]{0.42\textwidth}
		\centering
		\includegraphics[width=\textwidth]{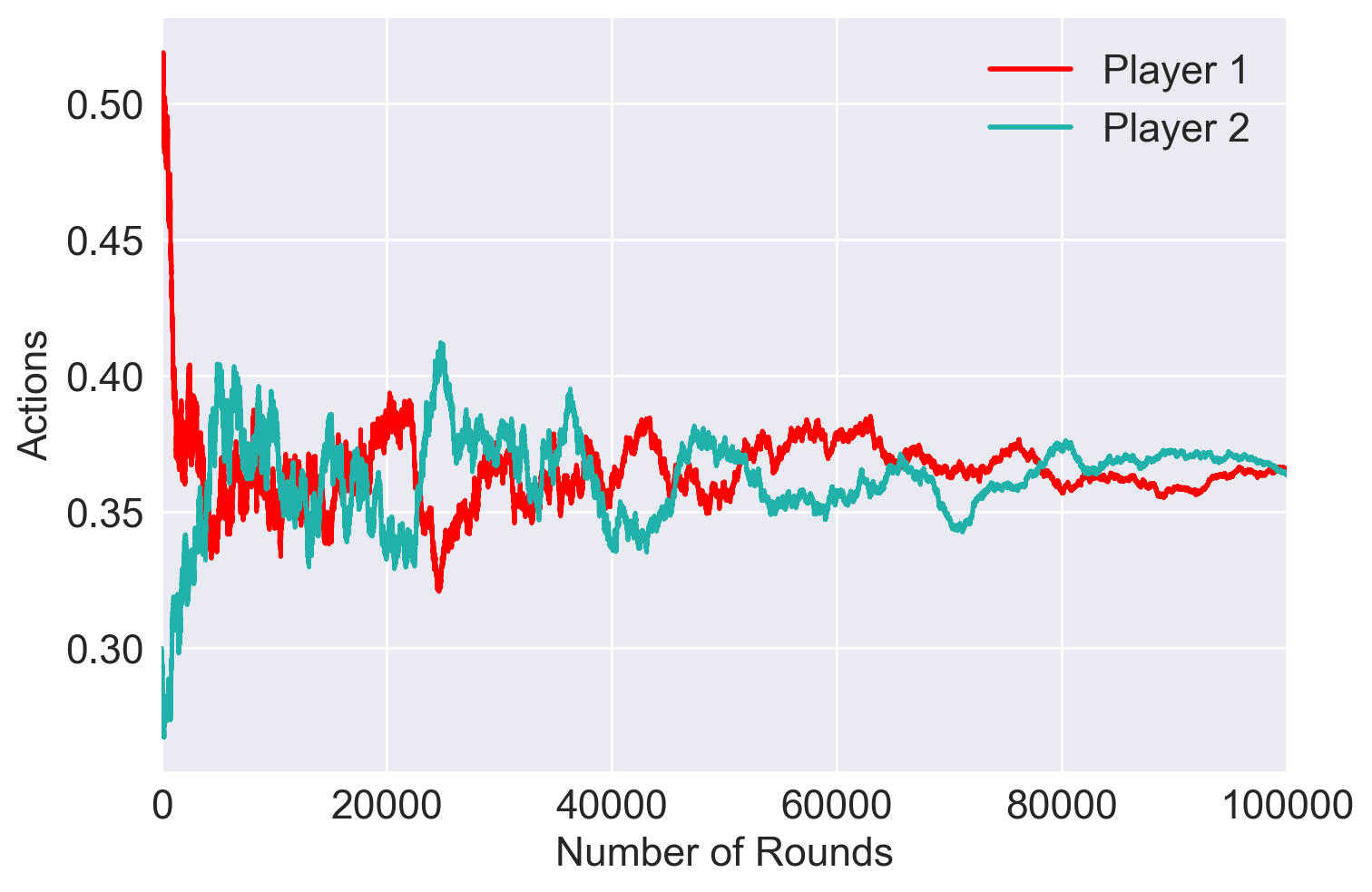}
		\caption{G4}
		\label{fig:g4}
	\end{subfigure}
	\caption{Convergence behavior of policy gradient in stochastic Cournot games: (a)-(b) are games with linear price and (c)-(d) are two-player games with general price functions.}
	\label{fig:convergence}
\end{figure}

\subsection{Investigative Studies}
\label{sec:investigative}
In this section, we provide two investigative studies relating to system performance under more general setups: 1) multi-agent Cournot game with non-linear price function; 2) hetergenous players that do not follow policy gradient updates. Note that our theoretical result in Section~\ref{sec:theorem} does not apply to the following two cases. $\textbf{G5:}$ three-player with quadratic price function $p(\bm{x}) = 1-(x_1+x_2+x_3)^2$ and no cost. $\textbf{G6:}$ three-player with linear price function $p(\bm{x}) = 1-(x_1+x_2+x_3)$ and no individual cost. One player does not follow policy gradient updates.
\begin{figure}
	\centering
	\begin{subfigure}[b]{0.42\textwidth}
		\centering
		\includegraphics[width=\textwidth]{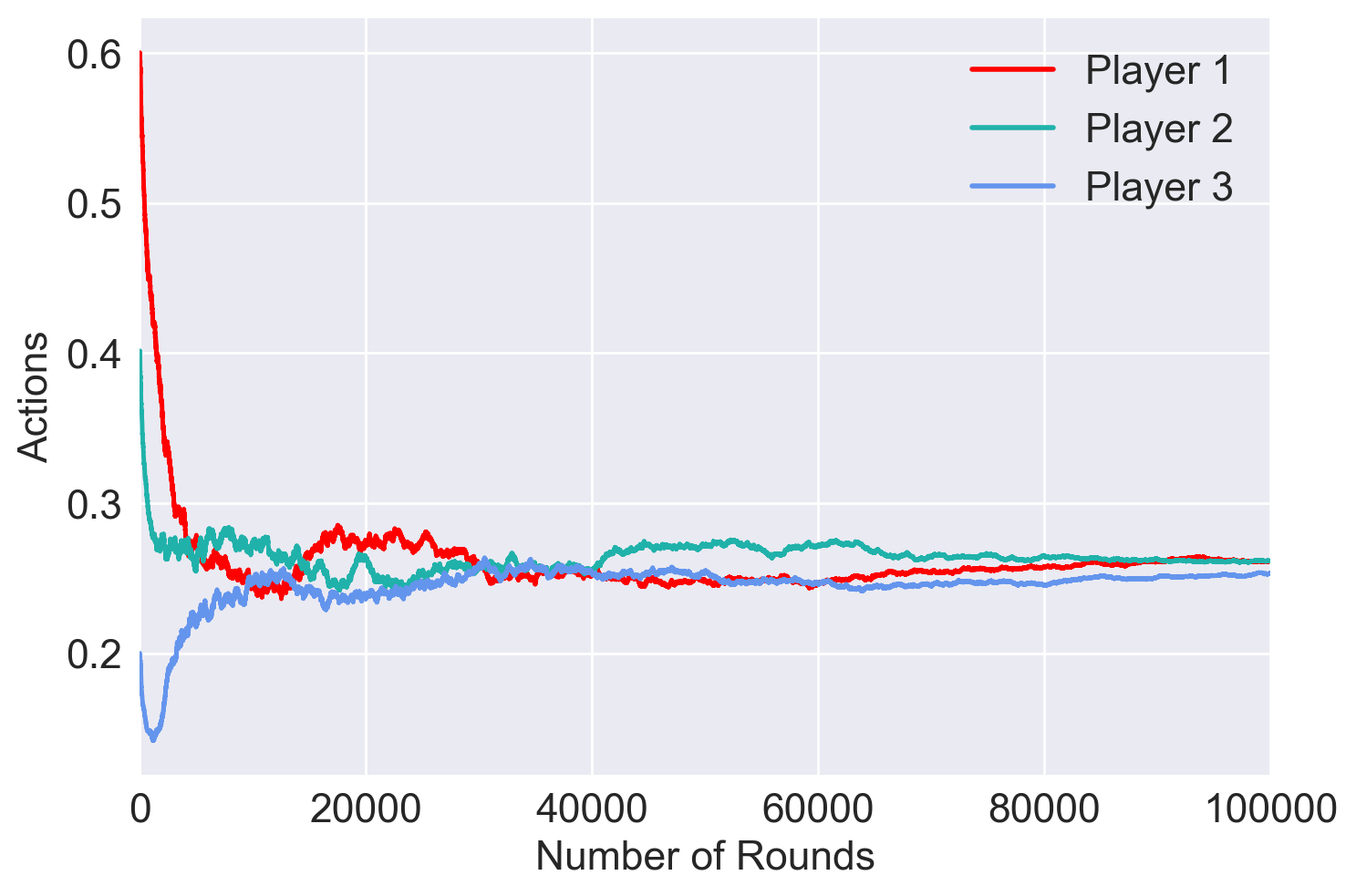}
		\caption{G5}
	\end{subfigure}
	\begin{subfigure}[b]{0.42\textwidth}
		\centering
		\includegraphics[width=\textwidth]{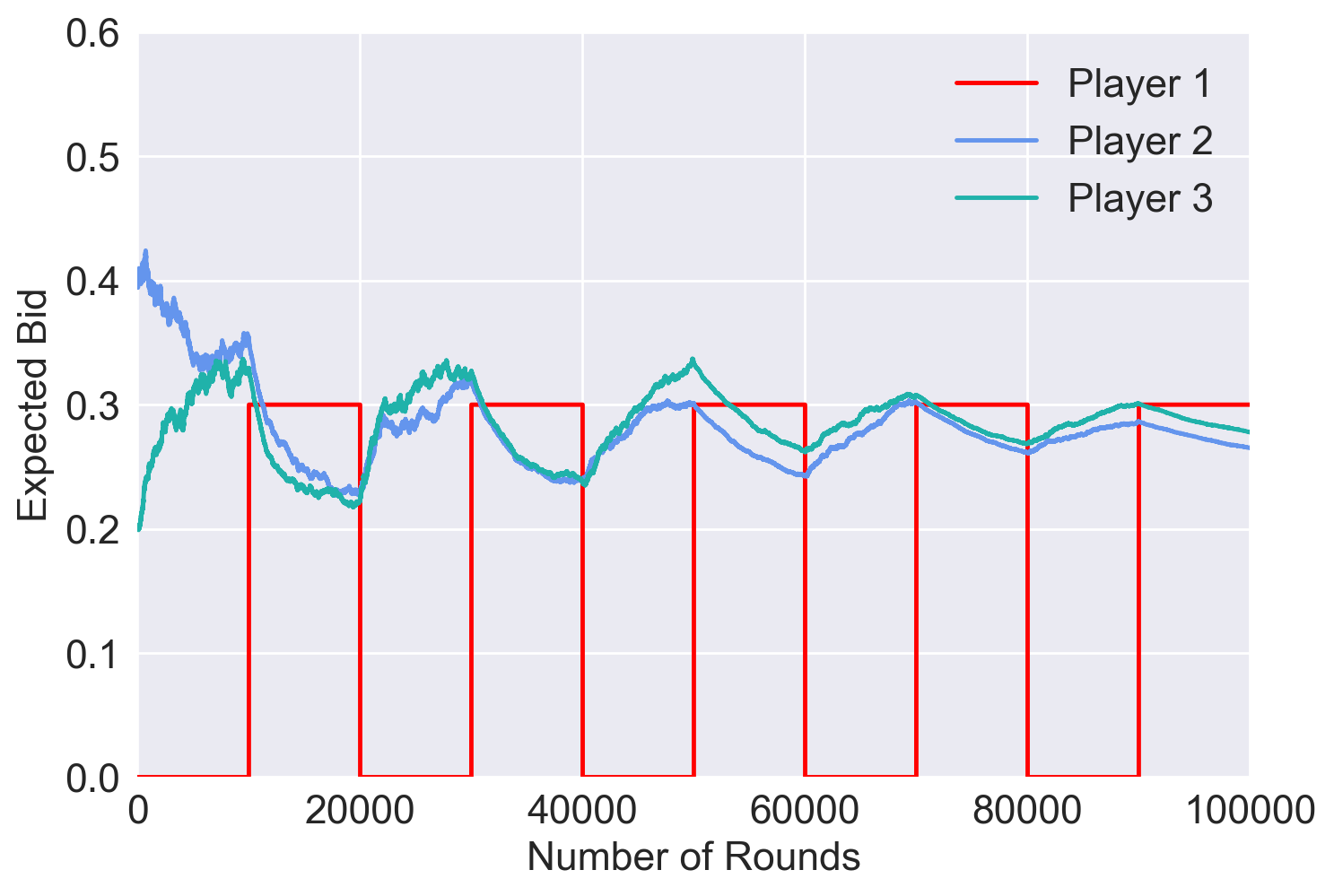}
		\caption{G6}
		\label{fig:g6}
	\end{subfigure}
	\caption{Dynamics of policy gradient beyond the convergence condition provided in Theorem 1.}
	\label{fig:robustness}
\end{figure}

Fig~\ref{fig:robustness} shows that both the three-player general price and heterogeneous players cases also converge to some equilibria, though they do not satisfy the convergence conditions in Theorem~\ref{theorem: pg_conv}, 
These results are promising in the sense that our results might be able to generalize to a broader class of games, and theoretically proving these would be valuable future work. There may also be settings where players are malicious, but designing optimal adversarial tactics and the detection algorithms, by themselves are topics that contain a vast body of literature and is beyond the scope of this work.

\section{Conclusion}
\label{sec:conslusion}
In this paper, we study the interaction of strategic players in Cournot games with limited feedback. We proved the convergence of policy gradient reinforcement learning to the Nash equilibrium, where player's policy is parameterized by the mean, under two conditions: either the price function is linear or there are two players.  Extending the results to more general conditions such as multi-player general price functions would be an important future direction.

\bibliographystyle{IEEEtran}
\bibliography{refs}
\end{document}